\documentclass{amsart}

\usepackage{amsmath,amsthm,amsfonts, bbm, amssymb, hyperref,graphicx,color}

\newcommand{\C}{\mathbb{C}}
\newcommand{\Hyp}{\mathbb{H}}
\newcommand{\R}{\mathbb{R}}
\newcommand{\Z}{\mathbb{Z}}

\newcommand{\Isom}{\text{Isom}}

\newtheorem{thm}{Theorem}[section]
\newtheorem{prop}[thm]{Proposition}
\newtheorem{lemma}[thm]{Lemma}
\newtheorem{cor}[thm]{Corollary}
\newtheorem{conv}[thm]{Convention}

\theoremstyle{definition}
\newtheorem{defi}[thm]{Definition}
\newtheorem{example}[thm]{Example}
\newtheorem{prob}[thm]{Problem}

\theoremstyle{remark}
\newtheorem{remark}[thm]{Remark}

\numberwithin{equation}{section}

\newcommand{\st}{\;:\;}
\newcommand{\mst}{\;:\;}

\begin{document}

\title[Quasi-isometric co-Hopficity of non-uniform lattices]{Quasi-isometric co-Hopficity of non-uniform lattices in rank-one semi-simple Lie groups}

\author[I.~Kapovich]{Ilya Kapovich}
\thanks{}

\address{
Department of Mathematics\\
University of Illinois at  Urbana-Champaign\\
1409 West Green Street\\
Urbana, IL 61801, USA}
\urladdr{http://www.math.uiuc.edu/\~{}kapovich/} 
\email{kapovich@math.uiuc.edu}
\thanks{The first author was supported by the NSF
  grant DMS-0904200.}

\author[A.~Lukyanenko]{Anton Lukyanenko}
\address{
Department of Mathematics\\
University of Illinois at  Urbana-Champaign\\
1409 West Green Street\\
Urbana, IL 61801, USA}
\urladdr{http://lukyanenko.net}
\email{anton@lukyanenko.net}
\thanks{The authors acknowledge support from the National Science Foundation grant DMS-1107452 ``RNMS: Geometric structures and representation varieties''.}


\begin{abstract}
We prove that if $G$ is a non-uniform lattice in a rank-one semi-simple Lie group $\ne \Isom( \Hyp^2_\R)$ then $G$ is quasi-isometrically co-Hopf.
This means that every quasi-isometric embedding $G\to G$ is coarsely surjective and thus is a quasi-isometry.
\end{abstract}

\maketitle

\section{Introduction}

The notion of co-Hopficity plays an important role in group
theory. Recall that a group $G$ is said to be \emph{co-Hopf} if $G$ is
not isomorphic to a proper subgroup of itself, that is, if every
injective homomorphism $G\to G$ is surjective. A group $G$ is \emph{almost
  co-Hopf} if for every injective homomorphism $\phi: G\to G$ we have
$[G:\phi(G)]<\infty$. Clearly, being co-Hopf implies being almost
co-Hopf. The converse is not true: for example, for any $n\ge 1$ the
free abelian group $\mathbb Z^n$ is almost co-Hopf but not co-Hopf.

It is easy to see that any freely decomposable group is not co-Hopf. In
particular, a free group of rank at least 2 is not co-Hopf. It is also
well-known that finitely generated nilpotent groups are always almost
co-Hopf and, under some additional restrictions, also co-Hopf~\cite{Be03}. An important result of Sela~\cite{Sela97} states that a torsion-free
non-elementary word-hyperbolic group $G$ is co-Hopf if and only if $G$
is freely indecomposable. Partial generalizations of this result are
known for certain classes of relatively hyperbolic
groups, by the work of Belegradek and Szczepa{\'n}ski~\cite{BeSz08}. Co-Hopficity has also been extensively studied for 3-manifold
groups and for Kleinian groups. Delzant and Potyagailo~\cite{DP03} gave a complete
characterization of co-Hopfian groups among non-elementary
geometrically finite Kleinian groups without 2-torsion. 

A counterpart agebraic notion is that of Hopficity. A group $G$ is
said to be \emph{Hopfian} if every surjective endomorphism $G\to G$ is
necessarily injective, and hence is an automorphism of $G$. This
notion is also extensively studied in geometric group theory. In particular, an
important result of Sela~\cite{Sela99} shows that every torsion-free
word-hyperbolic group is Hopfian. The notion of Hopficity admits a
number of interesting ``virtual'' variations. Thus a group $G$ is called
\emph{cofinitely Hopfian} if every endomorphism of $G$ whose image is
of finite index in $G$, is an automorphism of $G$, see, for example~\cite{BGHM}.

A key general theme in geometric group theory is the study of
``large-scale'' geometric properties of finitely generated groups.  Recall that if $(X,d_X)$ and $(Y,d_Y)$ are metric spaces, a map $f:X\to Y$ is called a \emph{coarse embedding} if there exist monotone non-decreasing functions $\alpha, \omega: [0,\infty)\to \mathbb R$ such that $\alpha(t)\le \omega(t)$, that $\lim_{t\to\infty} \alpha(t)=\infty$, and such that for all $x,x'\in X$ we have
\begin{align}
\alpha(d_X(x,x'))\le d_Y(f(x), f(x'))\le \omega(d_X(x,x')). \tag{*}
\end{align}
If $d_X$ is a path metric, then for any coarse embedding $f:X\to Y$ the function $\omega(t)$ can be chosen to be affine, that is, of the form $\omega(t)=at+b$ for some $a,b\ge 0$.

A coarse map $f$ is called a \emph{coarse equivalence} if $f$ is \emph{coarsely surjective}, that is, if there is $C\ge 0$ such that for every $y\in Y$ there exists $x\in X$ with $d_Y(y,f(x))\le C$. A map $f:X\to Y$ is called a \emph{quasi-isometric embedding} if $f$ is a coarse embedding and the functions $\alpha(t), \omega(t)$ in $(*)$ can be chosen to be affine,  that is, of the form $\alpha(t)=\frac{1}{\lambda}t -\epsilon$, $\omega(t)=\lambda t+\epsilon$ where $\lambda\ge 1$, $\epsilon\ge 0$. Finally, a map $f:X\to Y$ is a \emph{quasi-isometry} if $f$ is a quasi-isometric embedding and $f$ is coarsely surjective.

The notion of co-Hopficity has the following natural counterpart for metric spaces. We say that a metric space $X$ is \emph{quasi-isometrically co-Hopf} if every quasi-isometric embedding $X\to X$ is coarsely surjective, that is, if every quasi-isometric embedding $X\to X$ is a quasi-isometry. More generally, a metric space $X$ is called \emph{coarsely co-Hopf} if every coarse embedding $X\to X$ is coarsely surjective. Clearly, if $X$ is coarsely co-Hopf then $X$ is quasi-isometrically co-Hopf. If $G$ is a finitely generated group with a word metric $d_G$ corresponding to some finite generating set of $G$, then every injective homomorphism $G\to G$ is a coarse embedding. This easily implies that if $(G, d_G)$ is coarsely co-Hopf then the group $G$ is almost co-Hopf.

\begin{example}
The real line $\R$ is coarsely co-Hopf (and hence quasi-isometrically
co-Hopf). This follows from the fact that any coarse embedding must
send the ends of $\R$ to distinct ends. Since $\R$ has two ends, a
coarse embedding induces a bijection on the set of ends of $\R$. It is
then not hard to see that a coarse embedding from $\R$ to $\R$ must be coarsely surjective. See \cite{bridsonhaefliger} for the formal definition of ends of a metric space.
\end{example}

\begin{example}
\label{ex:rooted}
The rooted regular binary tree  $T_2$ is not quasi-isometrically co-Hopf. We can
identify the set of vertices of $T_2$ with the set of all finite
binary sequences. The root of $T_2$ is the empty binary sequence
$\epsilon$ and for a finite binary sequence $x$ its left child is the
sequence $0x$ and the right child is the sequence $1x$. Consider the
map $f: T_2\to T_2$ which maps $T_2$ isometrically to a copy of itself
that ``hangs below'' the vertex $0$. Thus $f(x)=0x$ for every finite
binary sequence $x$. Then $f$ is an isometric embedding but the image
$f(T_2)$ is not co-bounded in $T_2$ since it misses the entire
infinite branch located below the vertex $1$.

\end{example}

\begin{example}\label{ex:free}

Consider the free group $F_2 = F(a,b)$ on two
  generators. Then $F_2$ is not quasi-isometrically co-Hopf.

The Cayley graph $X$ of $F_2$ is a regular $4$-valent tree with every
edge of length $1$. We may view $X$ in the plane so that every vertex
has one edge directed upward, and three downward. Picking a vertex
$v_0$ of $X$, denote its left branch by $X_1$ and the remainder of the
tree by $X_2$. We have $X_1 \cup X_2 = X$, and $X_1$ is a rooted
ternary tree. Define a quasi-isometric embedding $f: X \rightarrow X$
by taking $f$ to be a shift on $X_1$ (defined similarly to Example \ref{ex:rooted})
and the identity on $X_2$. The map $f$ is not coarsely surjective, but
it is a quasi-isometric embedding. Moreover, for any vertices $x,x'$ of $X$ we
have $|d(f(x),f(x'))-d(x,x')|\le 1$.

One can also see that $F_2 = F(a,b)$ is not quasi-isometrically
co-Hopf for algebraic reasons. Let $u,v\in F(a,b)$ with $[u,v]\ne
1$. Then there is an injective homomorphism $h: F(a,b)\to F(a,b)$ such
that $h(a)=u$ and $h(b)=v$. This homomorphism $f$ is always a
quasi-isometric embedding of $F(a,b)$ into itself.

 If, in addition, $u$ and $v$ are chosen so that  $\langle u,
 v\rangle\ne F(a,b)$ then $[F(a,b): h(F(a,b))]=\infty$ and the image
 $h(F(a,b))$ is not co-bounded in $F(a,b)$.

Thus, the group $F_2$ is not almost co-Hopf and not quasi-isometrically co-Hopf.
\end{example}

\begin{example} There do exist finitely generated groups that are algebraically co-Hopf but not quasi-isometrically  co-Hopf. The simplest example of this kind is the solvable Baumslag-Solitar group $B(1,2)=\langle a,t| t^{-1}at=a^2\rangle$. It is well-known that $B(1,2)$ is co-Hopf. 

To see that $B(1,2)$ is not quasi-isometrically co-Hopf we use the fact that  $B(1,2)$ admits an isometric properly discontinuous co-compact action on a proper geodesic metric space $X$ that is ``foliated'' by copies of the hyperbolic plane $\Hyp^2_\R$.   We refer the reader to the paper of Farb and Mosher~\cite{FM98} for a detailed description of the space $X$, and will only briefly recall the properties of $X$ here. 

Topologically, $X$ is homeomorphic to the product $\mathbb R\times
T_3$ where $T_3$ is an infinite 3-regular tree (drawn upwards): there
is a natural projection $p: X\to T_3$ whose fibers are homeomorphic to
$\mathbb R$. The boundary of $T_3$ is decomposed into two sets: the
``lower boundary'' consisting of a single point $u$ and the ``upper
boundary'' $\partial_{\delta} X$ which is homeomorphic to the Cantor
set (and can be identified with the set of dyadic rationals). For any
bi-infinite geodesic $\ell$ in $T_3$ from $u$ to a point of
$\partial_{\delta} X$ the full-$p$-preimage of $\ell$ in $X$ is a copy
of the hyperbolic plane $\Hyp^2_\R$ (in the upper-half plane model).
The $p$-preimage of any vertex of $T_3$ is a horizontal horocycle in
the $\Hyp^2_\R$-``fibers''. Any two $\mathbb H^2$-fibers intersect
along a complement of a horoball in  $\Hyp^2_\R$.  

Similar to the above example for $F(a,b)$, we can take a
quasi-isometric embedding $f: T_3\to T_3$ whose image misses an
infinite subtree in $T_3$ and such that $|d(x,x')-d(f(x), f(x'))|\le 1$ for any vertices $x,x'$ of $T_3$. It is not hard to see that this map $f$ can be extended along the $p$-fibers to a map $\widetilde f: X\to X$ such that $\widetilde f$ is a quasi-isometric embedding but not coarsely surjective. Since $X$ is quasi-isometric to $B(1,2)$, it follows that $B(1,2)$ is not quasi-isometrically co-Hopf.
\end{example}

\begin{example}
Grigorchuk's group $G$ of intermediate growth provides another
intersting example of a group that is not quasi-isometrically
co-Hopf. This group $G$ is finitely generated and can be realized as a
group of automorphisms of the regular binary rooted tree $T_2$. The
group $G$ has a number of unusual algebraic properties: it is an
infinite 2-torsion group, it has intermediate growth, it is amenable
but not elementary amenable and so on. See Ch. VIII in \cite{DLH} for
detailed background on the Grigorchuk group. It is known that there
exists a subgroup $K$ of index 16 in $G$ such that $K\times K$
is isomorphic to a subgroup of index 64 in $G$. The map $K\to K\times
K$, $k\mapsto (k,1)$ is clearly a quasi-isometric embedding which is
not coarsely surjective. Since both $K$ and $K\times K$ are
quasi-isometric to $G$, it follows that $G$ is not quasi-isometrically
co-Hopf.

\end{example}

For Gromov-hyperbolic groups and spaces quasi-isometric co-Hopficity is closely related to the properties of their hyperbolic boundaries. We say that a compact metric space $K$ is \emph{topologically co-Hopf} if $K$ is not homeomorphic to a proper subset of itself. We say that $K$ is \emph{quasi-symmetrically co-Hopf} if every quasi-symmetric map $K\to K$ is surjective. Note that for a compact metric space $K$ being topologically co-Hopf obviously implies being quasi-symmetrically co-Hopf. 

\begin{example}
\label{ex:merenkov}
A recent important result of Merenkov~\cite{Merenkov10} shows that the converse implication does not hold. He constructed a round Sierpinski carpet $\mathbb S$ such that $\mathbb S$ is quasi-symmetrically co-Hopf. Since $\mathbb S$ is homeomorphic to the standard ``square'' Serpinski carpet, clearly $\mathbb S$ is not topologically co-Hopf.
\end{example}

It is well-known (see, for example, \cite{BoSch}) that if $X,Y$ are proper Gromov-hyperbolic geodesic metric spaces, then any quasi-isometric embedding $f: X\to Y$ induces a quasi-symmetric topological embedding $\partial f: \partial X\to \partial Y$ between their hyperbolic boundaries. It is then not hard to see that if $G$ is a word-hyperbolic group whose hyperbolic boundary $\partial G$ is quasi-symmetrically co-Hopf (e.g. if it is topologically co-Hopf), then $G$ is quasi-isometrically co-Hopf. This applies, for example, to any word-hyperbolic groups whose boundary $\partial G$ is homeomorphic to an $n$-sphere (with $n\ge 1$), such as fundamental groups of closed Riemannian manifolds with all sectional curvatures $\le -1$.

 The main result of this paper is the following:
 
 \begin{thm}\label{thm:A}
 Let $G$ be a non-uniform lattice in a rank-one semi-simple real Lie group other than $\Isom( \Hyp^2_\R)$. Then $G$ is quasi-isometrically co-Hopf.
 \end{thm}

Thus, for example, if $M$ is a complete finite volume non-compact hyperbolic
manifold of dimension $n\ge 3$ then $\pi_1(M)$ is quasi-isometrically
co-Hopf. Note that if $G$ is a non-uniform lattice in $\Isom( \Hyp^2_\R)$ then
the conclusion of Theorem~\ref{thm:A} does not hold since $G$ is a
virtually free group. 

If $G$ is a uniform lattice in a rank-one semi-simple real Lie group
(including possibly a lattice in $\Isom( \Hyp^2_\R)$) then $G$ is
Gromov-hyperbolic with the boundary $\partial G$ being homeomorphic to
$\mathbb S^n$ (for some $n\ge 1$). In this case it is easy to see that
$G$ is also quasi-isometrically co-Hopf since every topological
embedding from $\mathbb S^n$ to itself is necessarily surjective.

\begin{conv}\label{conv:X} From now on and for the remainder of this paper  let $X\ne
  \Hyp^2_\R$ be a rank-one negatively curved symmetric space with
  metric $d_X$ (or just $d$ in most cases). Namely, $X$ is isometric
  to a hyperbolic space $\Hyp^n_\R$ (with $n\ge 3$), $\Hyp^n_\C$ (with $n \geq 2$),
  $\Hyp^n_{H}$ over the reals, complexes, or quaternions, or to the
  octonionic plane $\Hyp^2_{\mathbb O}$.  
\end{conv}

If $G$ is as in Theorem~\ref{thm:A}, then $G$ acts properly
discontinuously (but with a non-compact quotient) by isometries on such a space $X$ and there exists a $G$-invariant collection $\mathcal B$ of disjoint horoballs in $X$ such that $(X\setminus \mathcal B) /G$ is compact.  The ``truncated'' space $\Omega=X\setminus \mathcal B$, endowed with the induced path-metric $d_\Omega$ is quasi-isometric to the group $G$ by the Milnor-Schwartz Lemma. Thus it suffices to prove that $(\Omega, d_\Omega)$ is quasi-isometrically co-Hopf.

Richard Schwartz~\cite{schwartz96} established quasi-isometric rigidity for non-uniform lattices in rank-one semi-simple Lie groups and we use his proof as a starting point. 

First, using coarse cohomological methods (particularly techniques of Kapovich-Kleiner~\cite{kapovich-kleiner2005}), we prove that spaces homeomorphic to $\mathbb R^n$ with ``reasonably nice'' metrics are coarsely co-Hopf.  This result applies to the Euclidean space $\mathbb R^n$ itself, to simply connected nilpotent Lie groups,  to the rank-one symmetric spaces $X$ mentioned above, as well as to the horospheres in $X$.
Let $f:(\Omega, d_\Omega) \to (\Omega, d_\Omega)$ be a quasi-isometric embedding. Schwartz' work implies that for every peripheral horosphere $\sigma$ in $\Omega$ there exists a unique peripheral horosphere $\sigma'$ of $X$ such that $f(\sigma)$ is contained in a bounded neighborhood of $\sigma'$. Using coarse co-Hopficity of horospheres,  mentioned above, we conclude that $f$ gives a quasi-isometry  (with controlled constants) between $\sigma$ and $\sigma'$. Then, following Schwartz, we extend the map $f$ through each peripheral horosphere to the corresponding peripheral horoball $B$ in $X$. We then argue that the extended map $\hat f: X\to X$ is a coarse embedding. Using coarse co-Hopficity of $X$, it follows that $\hat f$ is coarsely surjective, which implies that the original map  $f:(\Omega, d_\Omega) \to (\Omega, d_\Omega)$ is coarsely surjective as well.

It seems likely that the proof of Theorem~\ref{thm:A} generalizes to
an appropriate subclass of relatively hyperbolic groups. 
However, a more intriguing question is to understand what happens for higher-rank lattices:

\begin{prob}
Let $G$ be a non-uniform lattice in a semi-simple real Lie group of rank $\ge 2$. Is $G$ quasi-isometrically co-Hopf?
\end{prob}

Unlike the groups considered in the present paper, higher-rank lattices are not relatively hyperbolic. Quasi-isometric rigidity for higher-rank lattices is known to hold, by the result of Eskin~\cite{Eskin98}, but the proofs there are quite different from the proof of Schwartz in the rank-one case.

Another natural question is:

\begin{prob}
Let $G$ be as in Theorem~\ref{thm:A}. Is $G$ coarsely co-Hopf?
\end{prob}
Our proof only yields quasi-isometric co-Hopficity, and it is possible that coarse co-Hopficity actually fails in this context.

The result of Merenkov (Example \ref{ex:merenkov}) produces the first example of a compact metric space $K$ which is quasi-symmetrically co-Hopf but not topologically co-Hopf.
Topologically, $K$ is homeomorphic to the standard Sierpinkski carpet and there exists a word-hyperbolic group (in fact a Kleinian group) with boundary homeomorphic to $K$.
However, the metric structure on the Sierpinski carpet in Merenkov's example is not ``group-like'' and is not quasi-symmetric to the visual metric on the boundary of a word-hyperbolic group.

\begin{prob}
Does there exist a word-hyperbolic group $G$ such that $\partial G$ (with the visual metric) is quasi-symmetrically co-Hopf (and hence $G$ is quasi-isometrically co-Hopf), but such that $\partial G$ is not topologically co-Hopf? In particuar, do there exist examples of this kind where $\partial G$ is homeomorphic to the Sierpinski carpet or the Menger curve?
\end{prob}

The above question is particularly interesting for the family of hyperbolic buildings $I_{p,q}$ constructed by Bourdon and Pajot~\cite{Bourdon97,BP99}. In their examples $\partial I_{p,q}$ is homeomorphic to the Menger curve, and it turns out to be possible to precisely compute the conformal dimension of $\partial I_{p,q}$. Note that, similar to the Sierpinski carpet, the Menger curve is not topologically co-Hopf.

\begin{prob}
Are the Burdon-Pajot buildings $I_{p,q}$ quasi-isometrically co-Hopf? Equivalently, are their boundaries $\partial I_{p,q}$ quasi-symmetrically co-Hopf?
\end{prob}

It is also interesting to investigate quasi-isometric and coarse
co-Hopficity for other natural classes of groups and metric spaces.
In an ongoing work (in preparation), Jason Behrstock, Alessandro
Sisto, and Harold Sultan study quasi-isometric co-Hopficity for mapping class
groups and also characterize exactly when this property holds for
fundamental groups of 3-manifolds.

{\it Acknowledgement:} The authors would like to thank Misha Kapovich
for useful conversations. 

\section{Geometric Objects}
\subsection{Horoballs}

Recall that, by Convention~\ref{conv:X}, $X$ is a rank one symmetric
space different from  $\Hyp^2_\R$. Namely, $X$ is isometric to a
hyperbolic space $\Hyp^n_\R$ (with $n\ge 3$), $\Hyp^n_\C$ (with $n\ge 2$), $\Hyp^n_{H}$ over the reals, complexes, or quaternions, or to the octonionic plane $\Hyp^2_{\mathbb O}$. We recall some properties of $X$. See \cite{bridsonhaefliger}, Chapter II.10, for details.
\begin{defi}Let $0\in X$ be a basepoint and $\gamma$ a geodesic ray starting at $0$. The associated function $b: X \rightarrow \R$ given by
\begin{align} b(x) = \lim_{s \rightarrow \infty} d(x, \gamma(s)) - s\end{align}
is known as a Busemann function on $X$. A \textit{horosphere} is a level set of a Busemann function. The set $b^{-1}[t_0, \infty) \subset X$ is a \textit{horoball}. Up to the action of the isometry group on $X$, there is a unique Busemann function, horosphere, and horoball.
\end{defi}

A Busemann function $b(x)$ provides a decomposition of $X$ into \textit{horospherical coordinates}, a generalization of the upper-halfspace model. Namely, let $\sigma = b^{-1}(0)$ and decompose $X=\sigma \times \R^+$ as follows: given $x\in X$, flow along the gradient of $b$ for time $b(x)$ to reach a point $s \in \sigma$, and write $x=(s, e^{b(x)})$. In horospherical coordinates, the $\sigma$-fibers $\{s\}\times \R^+$ are geodesics, the $R^+$-fibers $\sigma \times \{t_0\}$ are horospheres, and the sets $\sigma \times [t_0, \infty)$ are horoballs. Other horoballs appear as closed balls tangent to the boundary $\sigma\times\{0\}$.

If $(M,d)$ is a metric space and $C\ge 0$, a path $\gamma: [a,b]\to M$,
parameterized by arc-length, is called a \emph{$C$-rough geodesic} in
$M$, if for any $t_1,t_2\in [a,b]$ we have 
\begin{align}
\big| d(\gamma(t_1), \gamma(t_2))- |t_1-t_2|\big|\le C.
\end{align}

If $Y, Y'$ are metric spaces, a map $f:Y\to Y'$ is \emph{coarsely
  Lipschitz} if there exists $C>0$ such that for any
$y_1,y_2\in Y$ we have $d_{Y'}(f(y_1),f(y_2))\le C d_Y(y_1,y_2)$. If
$Y$ is a path metric space then it is easy to see that  $f:Y\to Y'$ is
coarsely Lipschitz if and only if there exist constants $C, C'>0$ such that for any
$y_1,y_2\in Y$ with $d_Y(y_1,y_2)\le C$ we have $d_{Y'}(f(y),f(y'))\le
C'$.

The following two lemmas appear to be well known folklore facts:

\begin{lemma}
\label{lemma:horoball1}
There exists $C>0$ with the following property: Let $\mathcal B$ be a horoball in $X$, $x_1 \in X \backslash \mathcal B$ and $x_2 \in \mathcal B$. Let $b$ be the point in $\mathcal B$ closest to $x_1$. Then the piecewise geodesic $[x_1, b] \cup [b, x_2]$ is a $C$-rough geodesic.
\begin{proof}
Acting by isometries of $X$, we may assume that $\mathcal B$ is a fixed horoball that is tangent to the boundary of $X$ in the horospherical model. We may also assume that $b$ is the top-most point of $\mathcal B$, so that $x_1$ lies in the vertical geodesic passing through $b$. See Figure \ref{fig:horoball1}.

Consider the ``top'' of $\mathcal B$, i.e.\ the maximal subset of $\partial \mathcal B$ that is a graph in horospherical coordinates. Considering the Riemannian metric on $X$ in horospherical coordinates, one sees that the geodesic $[x_1, x_2]$ must pass through the top of $\mathcal B$.  Setting $C$ to be the radius of the top of $\mathcal B$, centered at $b$, completes the proof.
\end{proof}
\end{lemma}

\begin{figure}
\begin{minipage}[b]{0.48\linewidth}
\def\svgwidth{.95 \columnwidth}
\centering
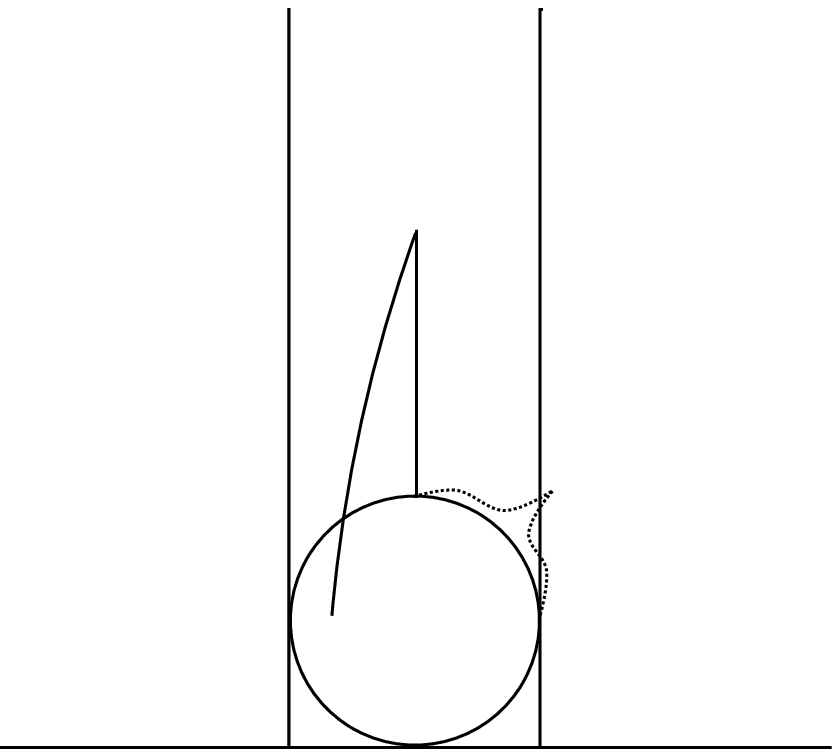
\caption{Lemma \ref{lemma:horoball1} 
for $X=\Hyp^2_\R$.\label{fig:horoball1}}
\label{fig:figure1}
\end{minipage}
\begin{minipage}[b]{0.48\linewidth}
\def\svgwidth{.95 \columnwidth}
\centering
 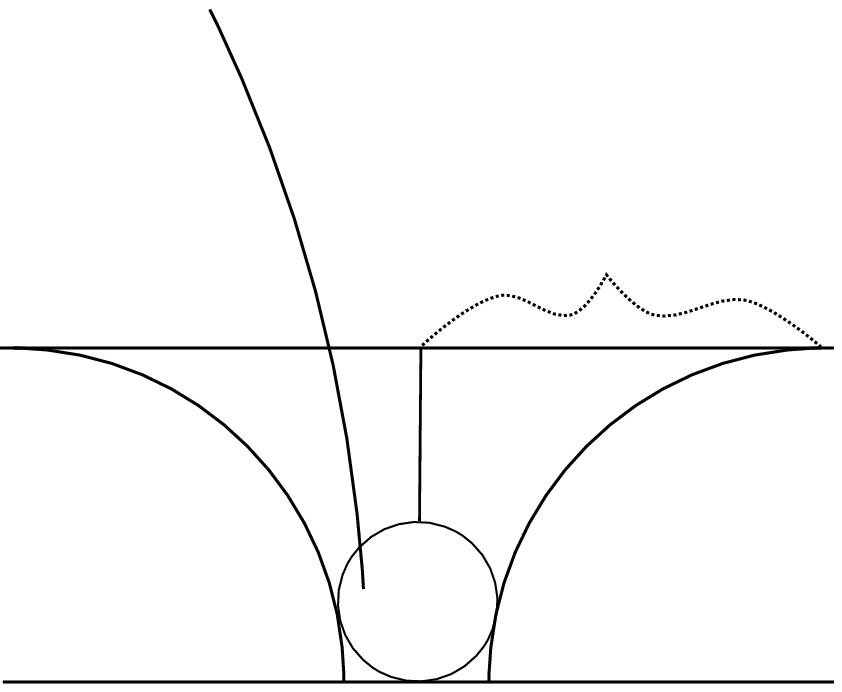
\caption{Lemma\ \ref{lemma:horoball2} for $X=\Hyp^2_\R$.
\label{fig:horoball2}}
\label{fig:figure2}
\end{minipage}
\end{figure}

\begin{lemma}
\label{lemma:horoball2}
Let $\mathcal B_1, \mathcal B_2$ be disjoint horoballs, and $x_1 \in \mathcal B_1, x_2 \in \mathcal B_2$. Let $[b_1, b_2]$ be the minimal geodesic between $\mathcal B_1$ and $\mathcal B_2$.  Then $[x_1, b_1] \cup [b_1, b_2] \cup [b_2, x_2]$ is a $C$-rough geodesic, for the value of $C$ in Lemma \ref{lemma:horoball1}.
\begin{proof}
The proof is analagous to that of Lemma \ref{lemma:horoball1}. We may normalize the horoballs $\mathcal B_1, \mathcal B_2$ as in Figure \ref{fig:horoball2}. The normalization depends only on the distance $d(\mathcal B_1,\mathcal B_2)$. Any geodesic $[x_1, x_2]$ must then pass through compact regions near $b_1$ and $b_2$.  Let $C(\mathcal B_1, \mathcal B_2)$ be the radius of this region in $\mathcal B_1$. Fixing $\mathcal B_1$ and varying $\mathcal B_2$, set $C = \sup C(\mathcal B_1, \mathcal B_2)$. The value $C(\mathcal B_1, \mathcal B_2)$ remains bounded if the distance between the horoballs goes to infinity (converging to the constant $C$ in Lemma \ref{lemma:horoball1}). Thus, the infimum is attained and $C<\infty$. This completes the proof.
\end{proof}
\end{lemma}

\begin{lemma}
\label{lemma:horoballs3}
Let $\mathcal B_1, \mathcal B_2$ be disjoint horoballs, $x_1 \in \mathcal B_1$, $x_2 \in \mathcal B_2$. Denote the minimal geodesic between $\mathcal B_1$ and $\mathcal B_2$ by $[b_1, b_2]$. Then $d(x_1, b_1) \leq d(x_1,x_2)$.
\begin{proof}
Fix $D>0$ and allow $\mathcal B_1, \mathcal B_2, x_1\in \mathcal B_1,$ and $x_2 \in \mathcal B_2$ to vary with the restriction $d(x_1, x_2) = D$. Define a function $f$ on the interval $[0,D]$ by
\[f(t) = \sup \{ d(x_1, b_1) \st d(\mathcal B_1, \mathcal B_2) = t\},\]
where the supremum is over all combinations of the variables with the restriction stated above, and $b_1$ denotes the closest point of $\mathcal B_1$ to $\mathcal B_2$. Then $f$ is a decreasing function, since increasing $t$ pushes the horoballs farther apart and forces $x_1$ closer to $x_2$. In particular, $f(D)=0$ since necessarily $x_1=b_1$. Conversely, $f(0)=D$, taking $x_2=b_1=b_2$. We then have for any choice of disjoint $\mathcal B_1, \mathcal B_2$ and $x_1, x_2$ in the corresponding horoballs, that 
\[d(x_1, b_1) \leq f(d(\mathcal B_1, \mathcal B_2) \leq d(x_1, x_2) = D,\]
as desired.
\end{proof}
\end{lemma}

\subsection{Truncated Spaces}
\begin{defi}
Let $X\ne \mathbb H^2_\R$ be a negatively curved rank one symmetric space. A \textit{truncated space} $\Omega$ is the complement in $X$ of a set of disjoint open horoballs. A truncated space is \textit{equivariant} if there is a (non-uniform) lattice $\Gamma \subset \Isom(X)$ that leaves $\Omega$ invariant, with $\Omega/\Gamma$ compact. 
\end{defi}

We will consider $\Omega$ with the induced path metric $d_\Omega$ from $X$. Under this metric, curvature remains negative in the interior of $\Omega$. The curvature on the boundary need not be negative. For an extensive treatment of truncated spaces, see \cite{schwartz96}.

\begin{remark}
Note that truncated spaces are, in general, not uniquely geodesic. Specifically, if $X$ is not a real hyperbolic space, then components of $\partial \Omega$ (which come from horospheres  in $X$) are isometrically embedded in $(\Omega, d_\Omega)$  copies of non-uniquely-geodesic Riemannian metrics on certain nilpotent groups. In particular, $(\Omega, d_\Omega)$  is not necessarily a $CAT(0)$-space.
\end{remark}

\begin{remark}
Let $X$ be a negatively curved rank one symmetric space and $\Gamma \subset \Isom(X)$ a non-uniform lattice. Then $X/\Gamma$ is a finite-volume manifold with cusps. In $X$, each cusp corresponds to a $\Gamma$-invariant family of horoballs. Removing the horoballs produces an equivariant truncated space $\Omega$ whose quotient $\Omega/\Gamma$ is the compact core of $X/\Gamma$.
\end{remark}

\begin{prop}
\label{prop:uniform}
Let $X$ be a negatively curved rank one symmetric space and $\Omega
\subset X$ an equivariant truncated space. Then the inclusion $\iota:
(\Omega, d_\Omega) \hookrightarrow (X, d_X)$ is a coarse embedding.
\end{prop}
\begin{proof}
Since $d_\Omega$ and $d_X$ are path metrics with the same line element, we have
\begin{align}d_X(x,y) \leq d_\Omega(x,y)\end{align}
To get the lower bound, define an auxilliary function
\begin{align}\beta(s) = \max \left\{ d_\Omega(x,y) \mst x,y \in \Omega \text{ and } d_X(x,y) \leq s \right\}.\end{align}
Let $K$ be a compact fundamental region for the action of $\Gamma$ on $\Omega$. Because $\Gamma$ acts on $\Omega$ by isometries with respect to both metrics $d_X$ and $d_\Omega$, we may equivalently define $\beta(s)$ by
\begin{align}\beta(s) = \max \left\{ d_\Omega(x,y) \mst x \in K, y \in \Omega \text{ and } d_X(x,y) \leq s\right\}.\end{align}
Because $K$ is compact and the metrics $d_X$, $d_\Omega$ are complete, $\beta(s) \in (0,\infty)$ for $s \in (0,\infty)$. Furthermore, $\beta: [0, \infty] \rightarrow [0, \infty]$ is continuous and increasing, with $\beta(0)=0$. Because horospheres have infinite diameter for both $d_X$ and $d_\Omega$ (they are isometric to appropriate nilpotent Lie groups with left-invariant Riemannian metrics, see \cite{schwartz96}), we also have $\beta(\infty) = \infty$. 

Let $\beta'$ be an increasing homeomorphism of $[0,\infty]$ with $\beta'(s) \geq \beta(s)$ for all $s$ and consider its inverse $\alpha(t)$. For $x,y \in \Omega$ we then have
\begin{align*} &d_\Omega(x,y) \leq \beta\left( d_X(x,y) \right) \leq \beta'\left( d_X(x,y) \right),\\
&\alpha\left(d_\Omega(x,y) \right) \leq d_X(x,y).\end{align*}
This concludes the proof.
\end{proof}

\begin{remark}
A more precise quantitative version of Proposition \ref{prop:uniform} can be obtained by studying geodesics in $\Omega$, see \cite{wordprocessingingroups}.
\end{remark}

\subsection{Mappings between truncated spaces}
For this section, let $\Omega \subset X$ be a truncated space, with $X \neq \Hyp^2_\R$, and $f: \Omega \rightarrow \Omega$ a $d_\Omega$-quasi-isometric embedding. To ease the exposition, we refer to the target truncated space as $\Omega' \subset X'$.

\begin{lemma}[Schwartz \cite{schwartz96}]
\label{lemma:schwartz1}
There exists $C>0$ so that for every boundary horosphere $\sigma$ of $\Omega$, there exists a boundary horosphere $\sigma'$ of $\Omega'$ such that $f(\sigma)$ is contained in a $C$-neighborhood of $\sigma'$.
\end{lemma}

Using nearest-point projection, we may assume $f(\sigma) \subset \sigma'$.

\begin{defi}
\label{defi:extension}
Let $\mathcal B, \mathcal B'$ be horoballs with boundaries $\sigma, \sigma'$. A point in $\sigma$ corresponds, in horospherical coordinates, to a geodesic ray in $B$. A map $\sigma \rightarrow \sigma'$  then extends to a map $\mathcal B \rightarrow \mathcal B'$ in the obvious fashion.

In view of Lemma \ref{lemma:schwartz1}, a $d_\Omega$-quasi-isometric embedding $f: \Omega \rightarrow \Omega'$ likewise extends to a map $f: X \rightarrow X'$ by filling the map on each boundary horoball.
\end{defi}

\begin{lemma}[Schwartz \cite{schwartz96}]
\label{lemma:schwartz2}
A quasi-isometry $f: \sigma \rightarrow \sigma'$ induces a quasi-isometry $\mathcal B \rightarrow \mathcal B'$, with uniform control on constants.
\begin{proof}[Idea of proof]
One considers the metric on the horospheres of $\mathcal B$ parallel to $\sigma$, or alternately fixes a model horosphere and varies the metric. One then shows that if $f$ is a quasi-isometry with respect to one of the horospheres, it is also a quasi-isometry with respect to the horospheres at other horo-heights. One then decomposes the metric on $\mathcal B$ into a sum of the horosphere metric and the standard metric on $\R$, in horospherical coordinates. This replacement is coarsely Lipschitz, so the extended map is also coarsely Lipschitz. Taking the inverse of $f$ completes the proof.	
\end{proof}
\end{lemma}

\section{Compactly Supported Cohomology}

\begin{defi}
Let $X$ be a simplicial complex and $K_i \subset X$ nested compacts with $\cup_i K_i = X$. Compactly supported cohomology $H^*_c(X)$ 	is defined by
\begin{align} H^*_c(X) = \varinjlim H^*(X, X \backslash K_i).
\end{align}
\end{defi}	

For a compact space $X$, $H^*_C(X)=H^*(X)$ but the two do not generally agree for unbounded spaces. We have $H^n_c(\R^n) = \mathbb Z$ and $H^n_c(\overline{\Omega})=0$ for a non-trivial truncated space $\Omega$. In fact, one has the following lemma.

\begin{lemma}
\label{lemma:everything}
Let $Z \subset \R^n$ be a closed subset. Then $H^n_c(Z) \neq 0$ if and only if $Z=\R^n$.
\begin{proof}
It is well-known that the choice of nested compact sets does not affect $H^n_c(Z)$. Choose the sequence $K_i = \overline{B(0,i)}\cap Z$, the intersection of a closed ball and $Z$.  With respect to the subset topology of $Z$, the boundary of $K_i$ is given by $\partial_ZK_i := \partial K_i \cap \partial \overline{B(0,i)}$. We have by excision 
$$H^n(Z, K_i) = H^n(K_i, \partial_Z K_i) = \widetilde H^n(K_i/\partial_Z K_i).$$
Note that $K_i \subset \overline{B(0,i)}$ and $\partial_Z K_i \subset \partial \overline{B(0,i)}$, so $K_i/\partial_Z K_i \subset \overline{B(0,i)}/\partial\overline{B(0,i)}$. Thus, if $K_i \neq B(0,i)$, then $K_i/\partial_ZK_i \subset S^n\backslash\{*\}$. That is, $K_i/\partial_ZK_i$ is a compact set in $\R^n$, and $\tilde H^n(K_i/\partial_ZK_i)=0$. Thus, if $Z=\R^n$, we have $H^n_c(Z)=\Z$. Otherwise, $H^n_c(Z)=0$.
\end{proof}
\end{lemma}

Compactly supported cohomology is not invariant under quasi-isometries or uniform embeddings. The remainder of this section is distilled from \cite{kapovich-kleiner2005}, where compactly supported cohomology is generalized to a theory invariant under uniform embeddings. For our purposes, the basic ideas of this theory, made explicit below, are sufficient.

\begin{defi}
Let $X$ be a simplicial complex with the standard metric assigning each edge length $1$. Recall that  a \textit{chain} in $X$ is a formal linear combination of simplices. The \textit{support} of a chain is the union of the simplices that have non-zero coefficients in the chain. The \textit{diameter} of a chain is the diameter of its support. 

An acyclic metric simplicial complex $X$ is $k$\textit{-uniformly acyclic} if there exists a function $\alpha$ such that any closed chain with diameter $d$ is the boundary of a $k+1$-chain of diameter at most $\alpha(d)$. If $X$ is $k$-uniformly acyclic for all $k$, we say that it is \textit{uniformly acyclic}.

Likewise, we say that a metric simplicial complex $X$ is $k$\textit{-uniformly contractible} if there exists a function $\alpha$ such that every continuous map $S^k \rightarrow X$ with image having diameter $d$ extends to a map $B^{k+1} \rightarrow X$ with diameter at most $\alpha(d)$. If $X$ is $k$-uniformly contractible for all $k$, we say it is \textit{uniformly contractible}.
\end{defi}

\begin{remark}
Rank one symmetric spaces and nilpotent Lie groups (with left-invariant Riemannian metrics) are uniformly contractible and uniformly acyclic.
\end{remark}

\begin{lemma}
\label{lemma:approx}
Let $X$, $Y$ be uniformly contractible and geometrically finite metric simplicial complexes and $f: X \rightarrow Y$ a uniform embedding. Then there exists an iterated barycentric subdivision of $X$ and $R>0$ depending only on the uniformity constants of $f, X,$ and $Y$ such that $f$ is approximated by a continuous simplicial map with additive error of at most $R$.
\begin{proof}
We first approximate $f$ by a continuous (but not simplicial) map by working on the skeleta of $X$. Starting with the $0$-skeleton, adjust the image of each vertex by distance at most $1$ so that the image of each vertex of $X$ is a vertex of $Y$. Next, assuming inductively that $f$ is continuous on each $k$-simplex of $X$, we now extend to the $k+1$ skeleton using the uniform contractibility of $Y$. Since error was bounded on the $k$-simplices, it remains bounded on the $k+1$-skeleton.

Now that $f$ has been approximated by a continous map, a standard simplicial approximation theorem replaces $f$ by a continuous simplicial map, with bounded error depending only on the geometry of $X$ and $Y$ (see for example the proof of Theorem 2C.1 of \cite{hatcher}).
\end{proof}
\end{lemma}

\begin{lemma}
\label{lemma:coho}
Let $X$ and $Y$ be uniformly acyclic simplicial complexes and $f: X \rightarrow Y$ a uniform embedding. Suppose furthermore that $f$ is a continuous simplicial map. Then if $H^n_c(X)\cong H^n_c(f X)$.
\begin{proof}
We first consturct a left inverse $\rho$ of the map $f_*: C_*(X) \rightarrow C_*(fX)$ induced by $f$ on the chain complex of $X$, up to a chain homotopy $P$. That is, $P$ will be a map $C_*(X) \rightarrow C_{*+1}(X)$ satisfying, for each $c\in C_*(X)$, the homotopy condition
\begin{align}\label{fla:homotopy}\partial P c = c - \rho f_* c - P \partial c\end{align}
and furthermore with diameter of $Pc$ controlled uniformly by the diameter of $c$.

We start with the $0$-skeleton. Each vertex $v'\in fX$ is the image of some vertex $v \in X$ (not necessarily unique). Set $\rho(v') = v$, and extend by linearity to $\rho: C_0(fX) \rightarrow C_0(X)$. To define $P$, let $v$ be an arbitrary vertex in $X$ and note that $\partial v=0$. We have to satisfy $\partial P v = v - \rho f_*v$. Since $X$ is acyclic, there exists a 1-chain $Pv$ satisfying this condition. Furthermore, note that $\rho f_*v$ is, by construction, a vertex such that $f(\rho f_*v) = f(v)$. Since $f$ is a uniform embedding, $d(\rho f_*v, v)$ is uniformly bounded above. Thus, $Pv$ may be chosen using \textit{uniform} acyclicity so that its diameter is also uniformly bounded above.

Assume next that $\rho$ and $P$ are defined for all $i<k$ with uniform control on diameters. Let $\sigma$ be a $k$-simplex in $X$. Then $\partial \rho f_*\sigma$ is a chain in $X$  whose diameter is bounded independently of $\sigma$. Then, by uniform acyclicity there is a chain $\sigma'$ with $\partial \sigma' = \partial \rho f_*\sigma$. We define $\rho(\sigma)= \sigma'$. As before, we need to link $\sigma'$ back to $\sigma$. We have
\begin{align} \partial ( \sigma - \sigma' - P\partial \sigma) = \partial \sigma - \rho f_* \partial \sigma - \partial P \partial \sigma.\end{align}
By the homotopy condition \ref{fla:homotopy}, we further have
\begin{align} \partial (\sigma - \sigma' - P\partial \sigma) = \partial \sigma - \rho f_* \partial  \sigma - ( \partial \sigma - \rho f_* \partial \sigma - P\partial \partial \sigma)=0.\end{align}
Thus, by bounded acyclicity there is a $k+1$ chain $P\sigma$ such that 
\begin{align}\partial P\sigma  = \sigma -\sigma' - P\partial \sigma,\end{align} as desired. We extend both $\rho$ and $P$ by linearity to all of $C_k(fX)$ and $C_k(X)$, respectively.

To conclude the argument, let $K$ be a compact subcomplex of $X$ and consider the complex $X/(X\backslash K) = K/\partial K$. The maps $P$ and $\rho \circ f_*$ on $C_*(X)$ induce maps on $C_*(K/\partial K)$, and the condition $\partial P c + P \partial c = c - \rho f_* c$ remains true for the induced maps and chains.

Because chain-homotopic maps on $C_*$ induce the same maps on homology, we have, for $h \in H_*(K/\partial K)$, $h = \rho f_* h$. Conversely, $f_*\rho$ is the identity on cell complexes, so still the identity on homology. Thus, $H_*(K/\partial K) \cong H_*(fK/\partial fK)$. By duality, $H^*(fK/\partial fK) \cong H^*(K/\partial K)$.

Taking $K_i$ to be an exhaustion of $X$ by compact subcomplexes and taking a direct limit, we conclude that $H^*_c(X) \cong H^*_c(fX)$.
\end{proof}
\end{lemma}

\begin{cor}
\label{cor:isomorphism}
Let $X$ and $Y$ be uniformly acyclic simplicial complexes and $f: X \rightarrow Y$ a uniform embedding. There exists an $R>0$ depending only on the uniformity constants of $f, X$, and $Y$ so that $H^n_c( N_R(fX)) \cong H^n_c(X)$. 
\begin{proof} Lemma \ref{lemma:approx} approximates $f$ by a continuous simplicial map, within uniform additive error. Lemma \ref{lemma:coho} shows that the resulting approximation induces an isomorphism on compactly supported cohomology.
\end{proof}
\end{cor}
\begin{thm}[Coarse co-Hopficity]
\label{thm:coarse} Let $(X, d_X)$ be a manifold homeomorphic to $\R^n$, with $d_X$ a path metric that is uniformly acyclic and uniformly contractible. For each pair of non-decreasing functions $\alpha, \omega: [0,\infty) \rightarrow \R$ with $\alpha(t)<\omega(t)$ and $\lim_{t\rightarrow \infty} \alpha(t)=\infty$, there exists a $C'$ such that any $(\alpha,\omega)$-coarse embedding $f: X \rightarrow X$ is $C'$-coarsely surjective.
\begin{proof} By Corollary \ref{cor:isomorphism}, there is a uniform $R>0$ such that  $H^n_c(N_R(fX)) \cong H^n_c(X) \cong \mathbb Z$. By Lemma \ref{lemma:everything}, $N_R(fX) =X$. Taking $C' = C+2R$ completes the proof.
\end{proof}
\end{thm}

\section{Main Result}

\begin{thm}[Quasi-Isometric co-Hopficity]
Let $\Omega \subset X$ and $\Omega' \subset X'$ be equivariant truncated spaces and $f: (\Omega, d_\Omega) \rightarrow (\Omega',d_{\Omega'})$ a quasi-isometric embedding. Then $f$ is coarsely surjective with respect to the truncated metric $d_\Omega$.
\begin{proof}
By Lemma \ref{lemma:schwartz1}, we may assume that $f$ maps boundary horospheres of $\Omega$ to boundary horospheres of $\Omega'$.  By Theorem \ref{thm:coarse}, $f$ is a surjection up to a constant independent of the boundary horosphere in question. We then have an extension $F: X \rightarrow X'$, as in Definition \ref{defi:extension}.

By Lemma \ref{lemma:schwartz2}, for each boundary horoball $\mathcal B$, the restriction $F\vert_{\mathcal B}$ is a quasi-isometry. By assumption, $F\vert_\Omega$ is a $d_\Omega$-quasi-isometry, so $F\vert_\Omega$  is a $d$-uniform embedding by Proposition \ref{prop:uniform}. Since $X$ is a path metric space, $F$ is then coarsely Lipschitz on all of $X$. 

We now show that $F$ is a uniform embedding by establishing a lower bound for distances between image points. Recall that all distances are measured with respect to $d=d_X$ unless another metric is explicitly mentioned.

Let $L\gg 2$ so that $F$ is coarsely $L$-Lipschitz and $F\vert_\mathcal B$ is coarsely $L$-co-Lipschitz for every boundary horoball $\mathcal B$.  Let $\alpha, \omega$ be increasing proper functions so that $f$ is an $(\alpha,\omega)$-uniform embedding.

Let $x_1, x_2 \in X$ with $d(x_1, x_2) \gg  0$. We need to provide a lower bound for $d(Fx_1, Fx_2)$ in terms of $d(x_1, x_2)$. Clearly, the lower bound will go to $\infty$ since $F$ is an isometry along vertical geodesics in horoballs. There are four cases to consider; in all cases we can ignore additive noise by working with sufficiently large $d(x_1, x_2)$ and slightly increasing $L$ .

\begin{enumerate}
\item Let $x_1, x_2 \in \mathcal B$ for the same horoball $\mathcal B$. Then $d(fx_1, fx_2) > d(x_1, fx_2)/L$.
\item Let $x_1, x_2 \in \Omega$. This case is controlled by the uniform embeddings $\Omega \hookrightarrow X$ and $\Omega' \hookrightarrow X'$ (Proposition \ref{prop:uniform}) and the $d_\Omega$-quasiisometry constants of $f$.
\item Let $x_1 \in \Omega, x_2 \in \mathcal B$ for a horoball $\mathcal B$. Let $b \in B$ be the closest point to $x_1$. Then by Lemma \ref{lemma:horoball1}, $[x_1, b] \cup [b, x_2]$ is a $C$-quasi-geodesic for a universal $C$ depending only on $X$ and $X'$ (see also Figure \ref{fig:horoball1}). We consider two sub-cases:

Suppose that $d(x_1, b) > d(x_1,x_2)/L^3$.  Let $b' \in f \mathcal B$ be the closest point to $f x_1$. Then by definition of $b$, we have
\begin{align*}
d(f^{-1}b', x_1) \geq d(b,x_1) \geq d(x_1, x_2)/L^3.
\end{align*}
Using Lemma \ref{lemma:horoballs3}, we conclude
$$d(fx_1, fx_2) \geq d(b', fx_2) \geq \alpha(d(f^{-1}, x_2)) \geq \alpha( d(x_1, x_2)/L^3).$$

Suppose, instead, that $d(x_1, b) \leq d(x_1, x_2)/L^3$. Then we have the estimate $d(fx_1, fb) \leq d(x_1, x_2)/L^2$. We also have $d(x_2, b) \approx d(x_1, x_2)$, so $d(fx_2, fb) \geq d(x_1, x_2)/L$. Consider now $b'\in \mathcal B$, the closest point to $fx_1$. By Lemma \ref{lemma:horoball1}, $d(fb, fb') \leq d(fb, fx_1)$. Thus, $$d(fx_1, fx_2) \geq d(x_1, x_2)/L - d(x_1,x_2)/L^2.$$

\item Let $x_1 \in \mathcal B_1, x_2 \in \mathcal B_2$ be in disjoint horoballs. This case is identical to the previous one, except one uses Lemma \ref{lemma:horoball2} rather than \ref{lemma:horoball1}.
\end{enumerate}

We have then provided a lower bound for $d(Fx_1,Fx_2)$ for any pair of points $x_1, x_2 \in X$. Thus, the extended map $F$ is a coarse embedding. By Theorem \ref{thm:coarse}, $F$ is then coarsely surjective. Namely, there exists $R>0$ so that $N_R(F(X))=X'$ (the neighborhood is taken with respect to $d$). 

We now show that the coarse surjectivity of $F$ with respect to $d$ implies the coarse surjectivity of $f$ with respect to $d_\Omega$.

Let $\omega' \in \Omega'$ be an arbitrary point. Since $F$ is coarsely surjective, there exists $x \in X$ so that $d_{X'}(f(x), \omega') \leq R$. If $x \in \Omega$, then we have shown that $\omega' \in N_R(f(\Omega))$. Otherwise, $x$ is contained in a horoball associated with $\Omega$. In appropriate horospherical coordinates, the horoball is given by $S \times (t_0, \infty)$ and $x$ can be written as $(s_1, t_1)$, with $t_1 > t_0$. Likewise, $f(x)$ has coordinates $(s'_1, t'_1)$, with $(t'_1 > t'_0)$. Furthermore, we have $f(s_1, t_0) = (s'_1, t_0)$. Now, $\omega' \in \Omega'$, so it has horospherical coordinates $(s'_2, t'_2)$ with $t'_2 < t'_0$. It is easy to see that 
\begin{align}R \geq d_{X'}( \omega', (s'_1, t'_1)) \geq d_{X'}( \omega', (s'_1, t'_0)) \\= d_{X'}(\omega', f(s_1, t_0) \geq d_{X'}(\omega', f(\Omega)).\nonumber
\end{align}
Thus, for an arbitrary $\omega' \in \Omega'$ we have $d_{X'}(\omega', f(\Omega)) \leq R$. Because $\Omega' \hookrightarrow X'$ is a uniform embedding, this implies that $f: \Omega \rightarrow \Omega'$ is coarsely surjective.
\end{proof}
\end{thm}

\bibliographystyle{amsplain}

\def\cprime{$'$}
\providecommand{\bysame}{\leavevmode\hbox to3em{\hrulefill}\thinspace}
\providecommand{\MR}{\relax\ifhmode\unskip\space\fi MR }
\providecommand{\MRhref}[2]{%
  \href{http://www.ams.org/mathscinet-getitem?mr=#1}{#2}
}
\providecommand{\href}[2]{#2}

\end{document}